\documentclass{article}
\makeatletter

\usepackage{amsmath, amsfonts, amssymb, amsthm} \usepackage{etoolbox}

\usepackage{mathtools}

\usepackage[utf8]{inputenc}
\usepackage[T1]{fontenc} \usepackage{lmodern} \usepackage{fix-cm}\rmfamily\DeclareFontShape{T1}{lmr}{bx}{sc}{<->ssub*cmr/bx/sc}{} \usepackage{microtype} \usepackage{csquotes}\MakeOuterQuote{"}  

\DeclareMathSymbol{\epsilon}{\mathord}{letters}{"22}

\DeclareMathSymbol{\phi}{\mathord}{letters}{"27}

\usepackage[a4paper,top=2cm,bottom=2cm,left=3cm,right=3cm,marginparwidth=2cm]{geometry}

\usepackage[colorlinks=true, allcolors=black]{hyperref} 
\usepackage[
  backend=biber,
  date=year, giveninits=true, abbreviate=true,
  sorting=nyt,
  maxbibnames=99,
  defernumbers=true,
]{biblatex}

\DeclareMultiCiteCommand{\cites}[\mkbibbrackets]{\cite}{\addsemicolon\space}

\AtEveryBibitem{\clearfield{issn}
  \clearfield{isbn}
  \clearfield{doi}
  \clearfield{urlyear,urlmonth} \ifentrytype{thesis}{
    \clearfield{url}
    \clearfield{eprint}
    \clearfield{eprinttype}
    }{}
  \ifentrytype{incollection}{
    \clearfield{url}
    }{}
  \ifentrytype{book}{
    \clearlist{mvbook}
    \clearfield{url}
    }{}
}

\DeclareFieldFormat{pages}{#1} \DeclareFieldFormat[article]{title}{#1} \DeclareFieldFormat[article]{journaltitle}{\mkbibemph{#1}}
\DeclareFieldFormat{eprint}{\href{\thefield{url}}{\texttt{\strfield{eprinttype}:\thefield{eprint}}}}

\DeclareBibliographyDriver{online}{\printnames{author}\addperiod\space\printfield{title}\addcomma\space\printdate\addperiod\space\printfield{eprint}\iffieldundef{note}{}{\space\printfield{note}}\addperiod}

\ExplSyntaxOn
\cs_generate_variant:Nn \tl_if_in:nnTF { en }
\cs_generate_variant:Nn \str_count:n { e }

\newbibmacro*{pages-or-eid}[1]{\iffieldundef{pages}{}{\tl_if_in:enTF{\thefield{pages}}{\bibrangedash}
      {#1\printfield{pages}}
      {\iffieldnum{pages}
        {\int_compare:nNnTF{\str_count:e{\thefield{pages}}}>{4}
          {\addcomma\space\#\printfield{pages}}
          {#1\printfield{pages}}}
        {\addcomma\space\#\printfield{pages}}}
  }
}
\ExplSyntaxOff

\newbibmacro*{journal+novolume}{\printdate
  \iffieldundef{number}{}{\addcomma\space
    \printtext{no.\space}\printfield{number}}\usebibmacro{pages-or-eid}{\addcomma\space}
}

\newbibmacro*{journal+volume}{\mkbibbold{\printfield{volume}}\iffieldundef{number}{}{\mkbibparens{\printfield{number}}}\usebibmacro{pages-or-eid}{\addcolon}\addcomma\space
  \printdate
}

\DeclareBibliographyDriver{article}{\printnames{author}\addperiod\space
  \printfield{title}\addperiod\space
  \iffieldundef{shortjournal}
    {\printfield{journaltitle}}
    {\mkbibemph{\printfield{shortjournal}\isdot}}\addcomma\space
  \iffieldundef{volume}
    {\usebibmacro{journal+novolume}}
    {\usebibmacro{journal+volume}}\addperiod
  \iffieldundef{note}{}{\space
    \printfield{note}\addperiod}
}

\usepackage[noabbrev]{cleveref}

\usepackage[toc,page]{appendix}
\AddToHook{cmd/appendix/before}{\crefalias{section}{appendix}}

\AddToHook{enddocument/afterlastpage}{\immediate\write\@mainaux{\string\gdef\string\appendixcount{\number\value{section}}}}
\newcommand{\appendixheadingname}{\@ifundefined{appendixcount}{Appendix}{\ifnum\appendixcount>1 Appendices\else Appendix\fi}}

\usepackage[shortlabels]{enumitem} 

\newenvironment{alphabetize}[1][]{\enumerate[label=(\alph*),#1]}{\endenumerate}
\setlist{font=\normalfont} 
   
\usepackage[bottom]{footmisc} 
\reversemarginpar 

\usepackage{mleftright}

\NewDocumentCommand\DeclareDelimiter{mmm}
 {
  \NewDocumentCommand{#1}{som}{\def\temp{\ifblank{##3}{\:\cdot\:}{##3}}
  \IfNoValueTF{##2}{
        \IfBooleanTF{##1}{
            #2\temp#3
        }{
            \mleft\ifblank{#2}{.}{#2}\temp\mright\ifblank{#3}{.}{#3}
        }
    }
    {
        \def\size{\csname##2\endcsname}
        \ifblank{#2}{}{\mathopen{\size#2}}\temp
        \ifblank{#3}{}{\mathclose{\size#3}} }
  }
 }

\DeclareDelimiter{\pa}{\lparen}{\rparen}
  
\DeclareDelimiter{\bra}{\lbrack}{\rbrack}

\DeclareDelimiter{\cbra}{\{}{\}}
\DeclareDelimiter{\abs}{\lvert}{\rvert} \DeclareDelimiter{\floor}{\lfloor}{\rfloor}           
\mathtoolsset{centercolon}  

\let\Re\relax \DeclareMathOperator{\Re}{Re} \let\Im\relax \DeclareMathOperator{\Im}{Im}

\newcommand{\inv}{\frac{1}}     

\usepackage{derivative}
\renewcommand{\d}{\odif}
\derivset{\odv}[delims-eval=.|]
\derivset{\pdv}[delims-eval=.|]

\RenewDocumentCommand{\to}{ oe{_} }{\IfValueTF{#1}{
    \IfValueTF{#2}{\xrightarrow[#2]{#1}}{\xrightarrow{#1}}
  }{
    \IfValueTF{#2}{\xrightarrow[#2]{}}{\rightarrow}
  }
}

\newcommand{\N}{\mathbb N}
\newcommand{\Z}{\mathbb Z}
 
\newcommand{\R}{\mathbb R}
\newcommand{\C}{\mathbb C}

\newcommand{\oInt}{\oint\limits}

\usepackage[section]{placeins}

\usepackage{keytheorems} 

\newkeytheoremstyle{example}{
bodyfont=\normalfont,
spacebelow=10pt
}

\newkeytheoremstyle{theorem}{
headindent=6pt,
headfont=\scshape\bfseries,
notefont=\normalfont\mdseries,
bodyfont=\itshape,
notebraces={(}{)},
postheadspace=6pt,
spacebelow=10pt
}

\renewkeytheoremstyle{remark}{
spaceabove=4pt,
headfont=\itshape,
notefont=\mdseries, notebraces={(}{)},
headpunct={.},
bodyfont=\normalfont,
postheadspace=7pt,
spacebelow=10pt
}

\newkeytheoremstyle{proof}{
spaceabove=4pt,
headfont=\itshape,
notefont=\itshape, notebraces={}{},
headpunct={.},
bodyfont=\normalfont,
postheadspace=10pt,
qed = $\blacksquare$,
spacebelow=12pt
}

\makeatother

\newkeytheorem{definition}[name = Definition, numbered=unless unique, numberwithin=section]
\newkeytheorem{example}[name = Example, style = example, numberwithin=section]
\newkeytheorem{theorem}[name = Theorem, style = theorem, numberwithin=section]
\newkeytheorem{proposition}[name = Proposition, style = theorem, numberwithin=section]
\newkeytheorem{corollary}[name = Corollary, style = theorem, numberwithin=section]
\newkeytheorem{lemma}[name = Lemma, style = theorem, numberwithin=section]
\newkeytheorem{proof}[name = Proof, style = proof, numbered=no]
\newkeytheorem{remark}[name = Remark, style = remark]

\newcommand{\W}{\mathcal W}
\newcommand{\D}{\mathbb D}

\DeclareBibliographyCategory{add}

\title{When can a power series be analytically continued?}
\author{Kei Beauduin}
\date{}

\begin{document}

\maketitle

\begin{abstract}
A classical line of research characterizes analytic continuation of power series through analytic interpolation of their coefficients. Under suitable hypotheses, the growth of the interpolating function reflects the geometry of the continuation domain, and conversely. We survey the development of these results from Leau and Le Roy through Lindelöf, Carlson, Dufresnoy--Pisot and Arakelian. We formulate the main results in a unified notation using compactifications of the complex plane. In particular, we give a comprehensive treatment of Carlson's second continuation theorem, which has received little attention, and prove a refinement based on the Laplace transform.
\end{abstract}

\setcounter{tocdepth}{3}\tableofcontents

\section{Introduction}

The equivalence of holomorphicity and analyticity makes power series central objects in complex analysis. Yet a power series is inherently a local representation. This raises a natural inverse problem: given a power series, to what extent can one recover the global analytic function it represents? The identity and monodromy theorems partly address this question: \emph{if} the germ can be continued along every path in a simply connected domain, its continuation is unique and independent of the path \cite{conway1978}.

What these theorems do not determine is \emph{when} a given power series admits analytic continuation. One might expect that no sufficiently general characterization of analytic continuability is possible, and that the question must be settled case by case.

\begin{quote}
\emph{``Whether for a given curve and a given function element there is an analytic continuation along the curve can be a difficult question. Since no degree of generality can be achieved which justifies the effort, no existence theorems for analytic continuations will be proved. Each individual case will be considered by itself.''}

\hfill --- John B. Conway \cite[p.~214]{conway1978}
\end{quote}

This local, \emph{Weierstrassian} viewpoint is conceptually useful but practically limited. Analytic continuation does not generate the function beyond its disk of convergence: the germ already determines it on its full Riemann surface. The challenge is to find that surface and representations valid there. A largely overlooked body of results tackles precisely this problem.

The first such results were obtained in 1898 by Leau \cite{leau1898,leau1899a} and Le Roy \cite{leroy1898,leroy1898a,leroy1899}. In their respective memoirs \cite{leau1899,leroy1900}, they proved several analytic-continuation theorems following a common pattern: a power series
\begin{equation}\label{e:f}
    f(z) := \sum_{n=0}^\infty c_n z^n,
\end{equation}
admits analytic continuation to a given unbounded domain if its coefficients $(c_n)_{n\in\N}\subset\C$ can be interpolated, in the sense that
\begin{equation}\label{e:interp}
    \phi(n) = c_n, \quad n \in \N := \{0, 1, 2, \dots\},
\end{equation}
where $\phi$ is analytic in a suitable region and satisfies appropriate exponential-growth conditions. Conversely, the Taylor coefficients of a function $f$ analytic in an unbounded domain can be interpolated by a function $\phi$ satisfying corresponding regularity. An interpolant $\phi$ satisfying \cref{e:interp} is often called a \emph{coefficient function}\footnote{Bieberbach reintroduced this terminology \cite{bieberbach1955}, although it had appeared earlier in Poukka \cite{poukka1916}.}. Such an interpolant need not be unique, but Carlson's theorem gives an important uniqueness criterion (\Cref{p:carlson}). Mathematical physicists instead use the term \emph{Carlsonian interpolation} \cite{bessis1965,bros1996}. Without growth restrictions similar to that of Carlson's theorem, the mere existence of an interpolant carries little information about analytic continuation. The substance of these results lies in obtaining an interpolant with controlled growth, since that control determines a continuation domain for $f$.

Our aim is to survey results of this kind throughout the literature. Hille's 1929 bibliographical survey \cite{hille1929} collects many of the early references cited here. Further accounts of related results appear in \cite{bieberbach1921,hadamard1926,dienes1957,bieberbach1955,hille1962,arakelian1991,arakelian2003,henrici1991}.

Five conceptual advances deserve particular emphasis: the sufficient conditions of Leau and Le Roy, Wigert's converse, Lindelöf's asymptotic expansions, Carlson's use of the Minkowski support function, and the Dufresnoy--Pisot refinement based on the Laplace transform.

Our principal contribution is a modern treatment of these results using a unified structure and notation based on suitable compactifications of the complex plane $\C$. In particular, we treat comprehensively an especially general theorem due to Carlson (1914) which has received little treatment since its publication. Many special cases of Carlson's theorem have consequently been rediscovered. We also state and prove an extension using the ideas from Dufresnoy and Pisot.

\section{Analytic continuation with bounded singularities}\label{s:bounded}

This section surveys results for which $f$ vanishes at infinity on the Riemann sphere $\hat\C := \C \cup \{\hat\infty\}$ and has a bounded singularity set. In this setting, the function $\phi$ interpolating the coefficients of $f$ is entire and of \emph{exponential type}: the quantity
\[
\sigma := \inf\cbra{\varsigma\in\R : \phi(s) = O(e^{\varsigma |s|})} = \limsup_{|s|\to\infty} \frac{\log |\phi(s)|}{|s|}
\]
is called the \emph{exponential type} of $\phi$ and is finite. It is nonnegative unless $\phi \equiv 0$, in which case $\sigma = -\infty$. The finer \emph{indicator function}, quantifying directional growth, is defined by
\begin{equation}\label{e:indic}
    h_\phi(\theta) := \inf \cbra{h\in\R : \phi(r e^{i\theta}) = O(e^{h r})} = \limsup_{r\to\infty} \frac{\log |\phi(re^{i\theta})|}{r}.
\end{equation}
It takes values in $\bar\R := \R \cup\{-\infty, \infty\}$ and, for an entire function $\phi$, is bounded above by the exponential type of $\phi$. We set it to $\infty$ at angles where it is undefined.

Another growth measure for entire functions is the \emph{order}, defined by
\[
\rho := \inf\cbra{\varrho > 0 : \phi(s) = O(e^{|s|^\varrho})} = \limsup_{|s|\to\infty} \frac{\log \log(e+ |\phi(s)|)}{\log |s|}.
\]
The order is at most $1$ when $\phi$ is of exponential type. For these definitions and their basic properties, see \cite{boas1954,levin1996}. 

For $r\ge 0$, define the disk centered at the origin
\[
\D_r := r\D, \quad \D := \{z : |z| < 1\},
\]
and for $c\in\R$, denote the closed left half-plane by
\[
\Pi_c := \{z\in\C : \Re z \le c\}.
\]

\subsection{Classical results}

The starting point is the following sufficient condition for the analytic continuation of $f$, known as \emph{Leau's theorem} \cite[Chap.~II, §18--20]{leau1899}.

\begin{theorem}[Leau, 1899]\label{t:leau}
    If $\phi(s)$ is an entire function of at most exponential type $0$ satisfying \cref{e:interp}, then the series $f(z)$ of \cref{e:f} is analytically continuable to $\C\setminus\{1\}$.
\end{theorem}

The case in which $\phi$ merely has order less than $1$ is often the result attributed to Leau \cite[§9]{hadamard1926} and was rediscovered by Le Roy \cite[§26]{leroy1900}. Leau also proved separately the case of order $1$ and exponential type $0$ \cite[Chap.~II, §20]{leau1899}.\footnote{He imposed a condition on the coefficients of $\phi$ that later formulas relating coefficients and exponential type show to be equivalent to exponential type $0$ \cite{levin1996}.}

\begin{example}
    To illustrate \Cref{t:leau}, let $\phi$ be a polynomial. A simple calculation then shows that $f$ is a rational function whose only possible pole is at $1$.
\end{example}

Wigert's converse supplied the next conceptual step \cite{wigert1900}. It shows that the criteria of Leau and Le Roy do more than provide occasional sufficient conditions: with the necessary additional assumption that $f$ vanishes at infinity, they characterize the problem completely. This result is known as the \emph{Wigert--Leau theorem}.

\begin{theorem}[Wigert, 1900]\label{t:wigert}
    The following are equivalent:
    \begin{alphabetize}
    \item There exists an entire function $\phi(s)$ of exponential type at most $0$ satisfying \cref{e:interp}.
    \item The series $f(z)$ of \cref{e:f} is analytically continuable to $\hat\C\setminus\{1\}$ with $f(\hat\infty) = 0$.
    \item The series of \cref{e:f} can be written as $f(z) = F\pa{\inv{1-z}}$ where $F$ is an entire function and $F(0) = 0$.
    \end{alphabetize}
\end{theorem}
The implication from (a) to (b) is Leau's theorem (\Cref{t:leau}), together with the regularity at $\hat\infty$ that makes the converse possible. Since every isolated singularity is either essential or a pole, the equivalence of (b) and (c) is simply a reformulation. Wigert, who was unaware of Leau's work, proved the equivalence of (a) and (c). Faber rediscovered the result in 1902 \cite{faber1903a,faber1903}, in the form of the equivalence of (a) and (b).

Wigert's formulation raises the natural question of how the growth of the entire functions $\phi$ and $F$ is related. Faber compared their orders \cite{faber1911}, a result independently recovered in \cite{gelfond1929,slobodetzky1941}. Macintyre and Wilson obtained the corresponding result for their types \cite{macintyre1947}, rediscovered in \cite{yardimci2002}. Further refinements appear in \cite{wilson1961,alzugaray2001}.

Le Roy extended Leau's theorem to functions of nonzero exponential type \cite[§27]{leroy1900}. Poukka in 1916 \cite{poukka1916} and Hardy in 1917 \cite{hardy1920} found converses. In this setting, the continuation region of $f$ has a more intricate form, and condition (c) of \Cref{t:wigert} has no straightforward analog.

\begin{theorem}\label{t:poukka}
    Let $\sigma \in [0, \pi)$. The following are equivalent:
    \begin{alphabetize}
    \item There exists an entire function $\phi(s)$ of exponential type at most $\sigma$ satisfying \cref{e:interp}.
    \item The series $f(z)$ of \cref{e:f} is analytically continuable to $\hat\C\setminus\exp(\bar\D_\sigma)$ with $f(\hat\infty) = 0$.
    \end{alphabetize}
\end{theorem}

Carlson had already described \Cref{t:poukka} in his 1914 thesis \cite[§3, §8]{carlson1914} as a special case of a much more general theorem; see also \cite{carlson1921,bieberbach1955,levin1980,levin1996}. Carlson's decisive innovation was to use the Minkowski support function in polar coordinates to describe general continuation regions, defined for $K\subset\C$ by
\begin{equation}\label{e:supp}
    \kappa_K(\theta) := \sup_{p\in K} \Re(p e^{-i\theta}).
\end{equation}
The vertical width of $K$ is
\[
\sup_{p\in K} \Im p - \inf_{p\in K} \Im p = \kappa_K\pa{\frac\pi2} + \kappa_K\pa{-\frac\pi2} \ge 0.
\]
Thus the condition
\begin{equation}\label{e:vert}
    \kappa_K\pa{-\frac\pi2} + \kappa_K\pa{\frac\pi2} < 2\pi,
\end{equation}
expresses that $K$ is strictly contained in a horizontal strip of width $2\pi$.
\begin{theorem}[Carlson, 1914]\label{t:carlsonb}
    Let $K$ be a compact convex subset of $\C$ satisfying \cref{e:vert}. The following are equivalent:
    \begin{alphabetize}
    \item There exists an entire function $\phi(s)$ of exponential type satisfying \cref{e:interp} and $h_\phi(\theta) \le \kappa_K(-\theta)$ for all $\theta$.
    \item The series $f(z)$ of \cref{e:f} is analytically continuable to $\hat\C\setminus \exp(-K)$ with $f(\hat\infty) = 0$.
    \end{alphabetize}
\end{theorem}

\begin{example}\label{x:e^a}
    The simplest example is $\phi(s)=e^{as}$. Then $c_n=e^{an}$ and $f(z)=1/(1-e^az)$, whose pole is $e^{-a}$. Moreover,
    \[
    h_\phi(\theta)= \Re(a)\cos\theta - \Im(a)\sin\theta = \kappa_{\{a\}}(-\theta)
    \]
    so the corresponding set is $K=\{a\}$. This explains the map $K\mapsto\exp(-K)$, and the width restriction ensures that the exponential map is injective on $K$.
\end{example}

The hypothesis that $\phi$ is of exponential type does not follow solely from the bound on $h_\phi$; see, for example, the counterexample in \cite{newman1976}.

The case in which $K$ is a disk of radius $\sigma < \pi$ is the particular instance of \Cref{t:poukka}, where $\kappa_K(\theta) = \sigma$.

In 1912, Pringsheim \cite[§7]{pringsheim1912} gave a sufficient condition in the case
\[
K = \cbra{x+iy : |y| < \frac\pi{2}, \enspace \pa{1 - \inv{\varrho^2}}e^{-x} + e^x \le 2 \cos y}, \quad \varrho \ge 1.
\]
When $\varrho > 1$, $K$ is bounded and his result is a special case of \Cref{t:carlsonb}, where $\exp(-K)$ corresponds to the disk of radius $\varrho/(\varrho^2-1)$ centered at $\varrho^2/(\varrho^2-1)$. In the limiting case $\varrho = 1$, however, $K$ is unbounded and $\exp(-K) = \{z\in\C : \Re z\ge 1/2\}$. Carlson's second continuation theorem in \Cref{s:carlson} applies instead.

\medskip

\Cref{t:carlsonb} and \eqref{e:vert} are closely related to Carlson's well-known uniqueness theorem \cite[Thm.~C]{carlson1914}. We present the following modern form \cite{bieberbach1955}.
\begin{proposition}[Carlson, 1914]\label{p:carlson}
    If $\phi$ is entire and of exponential type in $\Re s \ge 0$, satisfies
    \begin{equation}\label{e:hvert1}
        h_\phi\pa{-\frac\pi2} + h_\phi\pa{\frac\pi2} < 2\pi
    \end{equation}
    and vanishes at the natural integers, then $\phi$ vanishes identically.
\end{proposition}
Consequently, among entire functions of exponential type in $\Re s\ge0$ satisfying a common bound such as
\begin{equation}\label{e:vertsym}
    h_\phi\pa{\pm\frac\pi2} < \pi,
\end{equation}
the interpolant $\phi$ of $(c_n)_{n\in\N}$ is unique, since the difference between any two such interpolants satisfies the assumptions of Carlson's theorem. This applies, in particular, to the interpolants in \Cref{t:carlsonb} for a fixed set $K$. Hardy gave an elementary proof of the original, equivalent case \eqref{e:vertsym} \cite{hardy1920}, and Riesz later simplified it \cite{riesz1920}.

\subsection{The Laplace transform and the Dufresnoy--Pisot theorem}\label{s:dufr}

In 1929, P\'olya introduced a framework for studying entire functions of exponential type through their complex Laplace transforms \cite[Chap.~2]{polya1929}. If $\phi(s)=\sum_{n=0}^\infty \phi_n s^n$, its Laplace transform is
\begin{equation}\label{e:Phi}
    \Phi(p) := \sum_{n=0}^\infty \frac{\phi_n n!}{p^{n+1}}.
\end{equation}
If $\phi$ has exponential type $\sigma$, then \cref{e:Phi} converges for $|p| > \sigma$, i.e., in $\C\setminus\bar\D_\sigma$. Using this framework, Dufresnoy and Pisot proved the following extension of Carlson's theorem, which no longer requires convexity \cite{dufresnoy1951}.

\begin{theorem}[Dufresnoy--Pisot, 1951]\label{t:dufr}
    Let $K$ be a compact subset of $\C$ with connected complement such that \cref{e:vert} holds. The following are equivalent:
    \begin{alphabetize}
    \item There exists an entire function $\phi(s)$ of exponential type satisfying \cref{e:interp} such that $\Phi(p)$ is analytically continuable to $\C\setminus K$.
    \item The series $f(z)$ of \cref{e:f} is analytically continuable to $\hat\C\setminus \exp(-K)$ with $f(\hat\infty) = 0$.
    \end{alphabetize}
\end{theorem}

\begin{example}
    Similarly to \Cref{x:e^a}, consider $\phi(s) = e^{as} + e^{bs}$ with $\abs{\Im(a-b)} < 2\pi$. Then
    \[
    f(z) = \inv{1-e^a z} + \inv{1-e^b z}, \quad \Phi(p) = \inv{p-a} + \inv{p-b}.
    \]
    In \Cref{t:carlsonb}, we would require a compact convex set $K$ containing $a$ and $b$, for example the segment $[a, b]$, and continuation of $f$ would be guaranteed outside $\exp(-[a, b])$.
\end{example}

Here, the locations of the singularities of $\Phi$ give information on the location of the singularities of $f$, and vice versa. The correspondence is suggested by the discrete and continuous transforms below
\begin{equation}
    f(e^{-p}) = \sum_{n=0}^\infty \phi(n) e^{-pn}, \quad \Phi(p) = \int_0^\infty \phi(s) e^{-ps} \d s, \quad \Re p > h_\phi(0).
\end{equation}

Dufresnoy and Pisot proved necessity using
\begin{equation}
    \phi(s) = \inv{2\pi i} \oInt_\Gamma f(e^{-p}) e^{ps} \d p,
\end{equation}
where $\Gamma$ encircles $K$. The Cauchy integral formula then shows that $\phi$ interpolates $(c_n)_{n\in\N}$. P\'olya had already established sufficiency \cite[§43]{polya1929}, using
\begin{equation}\label{e:fleroy}
    f(z) = \inv{2\pi i} \oInt_\Gamma \frac{\Phi(p)}{1-ze^p} \d p,
\end{equation}
which follows from \emph{Pincherle's formula} for the inverse Laplace transform,
\begin{equation}\label{e:pincherle}
    \phi(s) = \inv{2\pi i} \oInt_\Gamma \Phi(p) e^{ps} \d p.
\end{equation}
\Cref{e:pincherle} implies $h_\phi(\theta) \le \kappa_K(-\theta)$ and therefore shows that \Cref{t:dufr} implies \Cref{t:carlsonb}. Given $\phi$, one can choose the convex set $K = \W(h_\phi)^\dagger$, where
\begin{equation}\label{e:wulff}
    \W(h) := \bigcap_{\theta\in\Theta} e^{i\theta}\Pi_{h(\theta)}  = \cbra{z\in\C : \forall\theta\in\Theta,\, \Re(z e^{-i\theta}) \le h(\theta)}
\end{equation}
is the \emph{Wulff shape}, and the dagger denotes complex conjugation. For this choice of $K$, we have $h_\phi(\theta) = \kappa_K(-\theta)$. Thus indicator functions are precisely support functions. Furthermore, \emph{P\'olya's theorem} identifies $\W(h_\phi)^\dagger$ with the convex hull of the singularities of $\Phi$ \cite{levin1996,beauduin2026b}.

P\'olya was not the first to describe \cref{e:fleroy}: Le Roy had already mentioned the formula briefly \cite[§28]{leroy1900} and had partially described the direct implication of \Cref{t:dufr} in 1900. Hardy used Le Roy's techniques to prove \Cref{t:poukka} \cite{hardy1920}. Building on the same approach, Brickman independently rediscovered the Dufresnoy--Pisot theorem in 1965 \cite{brickman1965}.

An expansion similar to
\begin{equation}\label{e:fphi(n)}
    f(z) = \sum_{n=0}^\infty \phi(n) z^n, \quad |z| < e^{-h_\phi(0)},
\end{equation}
follows from \cref{e:fleroy}:
\begin{equation}\label{e:fphi(-n)}
    f(z) = -\sum_{n=1}^\infty \phi(-n) z^{-n}, \quad |z| > e^{h_\phi(\pi)}.
\end{equation}
For $K=\W(h_\phi)^\dagger$, the region $e^{-h_\phi(0)} \le |z| \le e^{h_\phi(\pi)}$ is the smallest closed annulus containing $\exp(-K)$. It lies outside the open regions of convergence guaranteed by \cref{e:fphi(n),e:fphi(-n)}, although either series may still converge at individual boundary points.

\section{Analytic continuation with unbounded singularities}\label{s:u}

This section treats the case in which $f$ has unbounded singularities. In analogy with \Cref{s:bounded}, we need a notion of $f$ vanishing at infinity away from the directions in which its singularities accumulate. This requires a different compactification of $\C$, the \emph{radial compactification},
\begin{equation}
    \tilde\C = \C \cup \cbra{\tilde \infty e^{i\theta} : \theta \in \Theta},
\end{equation}
where $\Theta := \R/2\pi \Z$ is the set of angles. We view $\arg$ as a function from $\C\setminus\{0\}$ to $\Theta$ and, for $A\subset\Theta$, define the sector
\[
S_A := \arg^{-1}(A) = \cbra{r e^{i\theta} : r > 0,\, \theta \in A},
\]
where $A$ is an interval in $\Theta$.

In the theorems below, the single point $\hat\infty$ is replaced by directional infinities $\tilde\infty e^{iA} := \{\tilde\infty e^{i\theta} : \theta\in A\}$. By exponential type at $\tilde\infty e^{iA}$, we mean that $\phi(s) = O(e^{h|s|})$ as $s \to \tilde\infty e^{i A}$. The latter convergence is understood in terms of neighborhoods. Equivalently, $|s| \to \infty$ and every accumulation point of $\arg s$ belongs to $A$. We say that $\phi$ is of exponential type for $\Re s>c$ if it is exponentially bounded in every smaller closed half-plane $\Re s\ge \gamma$, where the bound may depend on $\gamma>c$. Thus $s$ may tend to either imaginary infinity, but not to the boundary infinities "$c\pm i\infty$".

Our formulation is particularly convenient when $A$ is open. Asymptotics at $\tilde\infty e^{i(-\frac\pi2,\frac\pi2)}$, for example, are usually expressed by requiring the bound as $|s|\to\infty$ with $s\in S_{(-\frac\pi2+\epsilon,\frac\pi2-\epsilon)}$ for every $\epsilon>0$. Arakelian's notion of \emph{inner exponential type} \cite{arakelian1985} is the supremum of the exponential types of $\phi$ in these subsectors and therefore coincides with exponential type at $\tilde\infty e^{i(-\frac\pi2,\frac\pi2)}$ in our terminology.

We also use the analogous notion of \emph{polynomial type} for $f$:
\[
c := \inf\cbra{\gamma\in\R : f(z) = O(|z|^\gamma) \text{ as } z\to\tilde\infty e^{i A}} = \limsup_{z\to\tilde\infty e^{i A}} \frac{\log |f(z)|}{\log |z|}.
\]

\subsection{The Le Roy--Lindelöf theorem}\label{s:leroylind}

\begin{theorem}[Lindelöf, 1902]\label{t:lind02}
    If $\phi(s)$ is analytic and of exponential type at most $0$ in a right half-plane and satisfies \cref{e:interp}, then the series $f(z)$ of \cref{e:f} is analytically continuable to $\C\setminus [1, \infty)$.
\end{theorem}

Le Roy proved this theorem under the hypothesis that $\phi$ has order less than $1$ in a right half-plane \cite[§32]{leroy1900}. In \cites{leroy1898}[§35]{leroy1900}, he claimed an extension in which $\phi$ is of exponential type at most $0$ and need only be analytic in a narrow sector. P\'olya later refuted this claim \cite[p.~155]{bieberbach1955}. Lindelöf proved the theorem for exponential type at most $0$ under the correct right-half-plane hypothesis \cite{lindelof1902,lindelof1903}, using the \emph{Abel--Plana formula} \cite[§13.14]{hardy1949}
\begin{equation}\label{e:abelplana}
    f(z) = \frac{\phi(0)}{2} + \int_0^\infty \phi(s) z^s \d s + i \int_0^\infty \frac{\phi(it)z^{it} - \phi(-it)z^{-it}}{e^{2\pi t}-1} \d t,
\end{equation}
defined for some branch of the complex logarithm. Ford subsequently gave an independent proof of a more restrictive version of \Cref{t:lind02} \cite{ford1903}, using the closely related formula
\begin{equation}\label{e:fordlind}
    f(z) = \int_{c+i\infty}^{c-i\infty} \frac{\phi(s) z^s}{e^{2\pi is}-1} \d s, \quad c\in(-1, 0).
\end{equation}

Lindelöf had in fact first found this formula in a different context \cite{lindelof1902b,lindelof1903a}.\footnote{For a modern treatment, see \cite{flajolet2010}.} Building on Ford's insight, he then proved the following extension of \Cref{t:lind02} \cite[§53]{lindelof1905}, known as the \emph{Le Roy--Lindelöf theorem}.

\begin{theorem}[Lindelöf, 1905]\label{t:leroylind}
    For some $\sigma\in [0, \pi)$, if $\phi(s)$ is analytic and of exponential type $\sigma$ in a right half-plane and satisfies \cref{e:interp}, then $f(z)$ defined by \cref{e:f} is analytically continuable to $S_{(\sigma, 2\pi-\sigma)}$.
\end{theorem}
Ford later treated the case $\sigma = 0$ \cite{ford1910}. The case of a non-symmetric sector follows under suitable changes of variables \cite{mkrtchyan2026}.

A further conceptual advance was Lindelöf's connection between analytic continuation and asymptotic expansion at infinity. More precisely, he established the following finite asymptotic expansion for $f$ \cite[§55]{lindelof1905}. Let $c<0$ be nonintegral and suppose that $\phi$ is analytic and of exponential type $\sigma$ in a right half-plane containing $\{s : \Re s\ge c\}$. Then
\begin{equation}\label{e:lindasymp}
    f(z)= -\sum_{c \le n < 0} \phi(n) z^n + o(|z|^c) = -\sum_{n=1}^{\floor{-c}} \phi(-n) z^{-n} + o(|z|^c)
\end{equation}
as $|z|\to \infty$ in $S_{(\sigma, 2\pi-\sigma)}$. In particular, if $c<0$ can be chosen arbitrarily negative, then $f$ has a \emph{Poincaré asymptotic expansion}. Although it was discovered earlier, \cref{e:lindasymp} extends \cref{e:fphi(-n)} to cases in which $\phi$ cannot be chosen entire and of exponential type.

Wigert's converse, together with the similarity between Leau's theorem and the Le Roy--Lindelöf theorem, suggested that the latter should admit one as well. Oseen and Svensson made partial progress soon after Lindelöf \cite{oseen1906,svensson1908}, and Bernstein gave a complete converse in 1926 \cite{bernstein1926}.

\begin{theorem}\label{t:leroylindconv}
    Let $c\in [-1, 0)$, $R > 0$ and $\sigma\in [0, \pi)$. The following are equivalent:
    \begin{alphabetize}
    \item There exists an analytic function $\phi(s)$ of exponential type for $\Re s > c$ satisfying \cref{e:interp} and $h_\phi(\theta) \le \sigma \abs{\sin \theta} - \log(R) \cos \theta$ for $\theta\in [-\frac\pi2, \frac\pi2]$.
    \item The series $f(z)$ of \cref{e:f} is analytically continuable to $\D_R \cup S_{(\sigma, 2\pi-\sigma)}$ and is of polynomial type at most $c$ at $\tilde\infty e^{i(\sigma, 2\pi-\sigma)}$.
    \end{alphabetize}
\end{theorem}
As Bernstein later acknowledged \cite{bernstein1928}, Carlson had already proved this case in his 1914 thesis. In fact, Carlson's result was much broader and is the subject of the next section.

Bernstein's proof of the converse employs the following formula \cites{bernstein1926}:
\begin{equation}
    \phi(s) = \frac{\sin(\pi s)}{\pi} \int_0^{\infty e^{i\theta}} \pa{f(z) - \sum_{n=0}^{N-1} c_n z^n} (-z)^{-s-1} \d z,
\end{equation}
with $N\in\N$ and $\theta\in (\sigma, 2\pi-\sigma)$. The bound on $h_\phi$ is equivalent to
\[
|\phi(s)| \le R^{-\Re s} e^{\sigma\abs{\Im s} + o(|s|)} \quad\text{as}\quad s \to \tilde\infty e^{i[-\frac\pi2, \frac\pi2]}.
\]
Nevertheless, Bernstein's necessary and sufficient conditions have a slight mismatch between their assumptions and conclusions, so they do not yield a clean if-and-only-if statement. The reformulation in \Cref{t:leroylindconv} using radial compactification and polynomial type is equivalent to Bernstein's result, while allowing the additional generality $R\ne1$.

\medskip

A different kind of converse to the Le Roy--Lindelöf theorem was later found by Arakelian \cite[Thm.~1.1]{arakelian1985}.

\begin{theorem}[Arakelian, 1984]\label{t:arak84}
    Let $R > 0$ and $\sigma\in [0, \pi)$. The following are equivalent:
    \begin{alphabetize}
    \item There exists a function $\phi(s)$ analytic for $\Re s \ge 0$ and of exponential type at $\tilde\infty e^{i(-\frac\pi2,\frac\pi2)}$, satisfying \cref{e:interp} and $h_\phi(\theta) \le \sigma \abs{\sin \theta} - \log(R) \cos \theta$ for $\theta\in (-\frac\pi2, \frac\pi2)$.
    \item The series $f(z)$ of \cref{e:f} is analytically continuable to $\D_R \cup S_{(\sigma, 2\pi - \sigma)}$.
    \end{alphabetize}
\end{theorem}

Compared with \Cref{t:leroylindconv}, Arakelian's theorem gives no control over the growth of $f$. The distinction already appears for $f(z)=z$, which satisfies clause (b) of \Cref{t:arak84} with $R=1$ and $\sigma=0$, but has polynomial type $1$ at infinity and therefore cannot satisfy clause (b) of \Cref{t:leroylindconv}. In return, the interpolant $\phi$ need only be analytic in a smaller half-plane and have exponential type away from the directions $\pm i\tilde\infty$. The absence of control in those directions accommodates this example, but also prevents the applicability of Carlson's theorem (\Cref{p:carlson}), so the interpolant is not unique. Constructing $f$ from $\phi$ requires only a simple modification of \cref{e:fordlind}. Recovering $\phi$ from $f$, however, is more technical, since it invokes an approximation theorem to compensate for the missing growth condition. Arakelian further showed that clause (a) of \Cref{t:arak84} can be strengthened by requiring $\phi$ to be entire and of order at most $1$ \cite[Thm.~1.2]{arakelian1985}.

\subsection{A general continuation theorem of Carlson}\label{s:carlson}

If we distinguish the different parallel directions that lead to a directional infinity of $\tilde\C$, we arrive at the \emph{pararadial compactification} of the complex plane,
\[
\bar\C := \C \cup \cbra{(\infty + iy) e^{i\theta} : y\in\bar\R, \enspace \theta \in\R}.
\]
Its topology and convexity, and the extensions of the support function and exponential map used below were introduced in \cite{beauduin2026b} and are adopted throughout the ensuing sections. 

The next result, proved by Carlson in his thesis \cite[§4, §7]{carlson1914}, characterizes continuation outside an unbounded singularity set. It does not appear elsewhere in the literature, although several special cases were independently discovered much later. This may be due both to the limited availability of the thesis and the impenetrability of Carlson's exposition.

\begin{theorem}[Carlson, 1914]\label{t:carlsonu}
    Let $c\in [-1, 0)$ and let $K$ be a closed convex subset of $\bar\C$ such that $\kappa_K(0)$ is finite and \cref{e:vert} holds. The following are equivalent:
    \begin{alphabetize}
    \item There exists an analytic function $\phi(s)$ of exponential type for $\Re s > c$ satisfying \cref{e:interp} and $h_\phi(\theta) \le \kappa_K(-\theta)$ for $\theta\in [-\frac\pi2, \frac\pi2]$.
    \item The series $f(z)$ of \cref{e:f} is analytically continuable to $\tilde\C\setminus \exp(-K)$ and is of polynomial type at most $c$ there.
    \end{alphabetize}
\end{theorem}

Taking
\[
K = \cbra{p \in \bar\C : \Re p \le -\log R,\, \abs{\Im p} \le \sigma}
\]
gives $\exp(-K) = \cbra{z\in\tilde\C : |z|\ge R,\, \abs{\arg z} \le \sigma}$ and hence \Cref{t:leroylindconv} follows.

Hille's book contains an exposition of the direct implication in \Cref{t:carlsonu} \cite[Thm.~11.3.1]{hille1962}. In the proof, he also notes that $\C\setminus\exp(-K)$ extends to infinity in the limiting case
\begin{equation}\label{e:vert2}
    \kappa_K\pa{-\frac\pi2} + \kappa_K\pa{\frac\pi2} = 2\pi, \quad \kappa'_K\pa{-\frac\pi2}=-\infty \quad\text{and}\quad
    \kappa'_K\pa{\frac\pi2}=\infty.
\end{equation}
The support function always admits left and right derivatives $\kappa'_K(\theta^-) \le \kappa'_K(\theta^+)$ \cite{levin1980,beauduin2026b}. When $K$ is convex, one can show that $(\kappa_K(\theta) + i y) e^{i\theta} \in \partial K$ for every $y\in [\kappa'_K(\theta^-), \kappa'_K(\theta^+)]$. Nondifferentiability of the support function therefore corresponds to straight edges in $K$. Condition \cref{e:vert2} states that $K$ is contained in an open horizontal strip of width $2\pi$, has no horizontal boundary edges, approaches both boundary lines as $\Re p\to-\infty$ and reaches them at $-\infty \pm i \kappa_K(\pm \frac\pi2) \in \partial K$.

Under \cref{e:vert2}, clause (a) still implies that $f$ is analytically continuable to $\C\setminus\exp(-K)$. Carlson notes, however, that no corresponding conclusion about the growth of $f$ at infinity is possible in general.

The following parallel result extends \Cref{t:arak84}.

\begin{theorem}[Gawronski--Trautner, 1976; Arakelian, 1992]\label{t:gta}
    Let $K$ be a closed convex set such that $\kappa_K(0)$ is finite and \cref{e:vert} holds. The following are equivalent:
    \begin{alphabetize}
        \item For every $\eta\in(0,\pi/2)$, there exists an entire function $\phi_\eta(s)$ of exponential type satisfying \cref{e:interp} and $h_{\phi_\eta}(\theta) \le \kappa_K(-\theta)$ for $\theta\in [-\eta, \eta]$.
        \item There exists an entire function $\phi(s)$ of order at most $1$ satisfying \cref{e:interp} and $h_\phi(\theta) \le \kappa_K(-\theta)$ for $\theta\in (-\frac\pi2, \frac\pi2)$.
        \item The series $f(z)$ of \cref{e:f} is analytically continuable to $\C\setminus\exp(-K)$.
    \end{alphabetize}
\end{theorem}

Gawronski and Trautner proved the equivalence of (a) and (c) \cite[Thms.~1, 3 i)]{gawronski1976}. Arakelian later established the equivalence of (b) and (c), using approximation theorems \cite[Thm.~3.1]{arakelian1992}.

Under an additional growth hypothesis, Gawronski and Trautner obtained an interpolant with stronger properties \cite[Thm.~3 ii)]{gawronski1976}. This construction corresponds to Carlson's converse in \Cref{t:carlsonu}.

\subsection{Refinements involving the Laplace transform}\label{s:ulapl}

The following extension of \Cref{t:carlsonu} has two features. It incorporates the Dufresnoy--Pisot refinement discussed in \Cref{s:dufr} and allows $K$ to be a subset of $\bar\C$, possibly a purely infinite set, meaning that $K$ is disjoint from $\C$. This happens when $\kappa_K(0) = -\infty$, and in that case, $f$ is entire.

\begin{theorem}\label{t:carlsonlapl}
    Let $c\in [-1, 0)$, and let $K$ be a closed subset of $\bar\C$ with connected complement such that $\kappa_K(0)<\infty$, and \cref{e:vert} holds. The following are equivalent:
    \begin{alphabetize}
    \item There exists an analytic function $\phi(s)$ of exponential type for $\Re s > c$ satisfying \cref{e:interp} and such that $\Phi(p)$ is analytically continuable to $\bar\C\setminus K$.
    \item The series $f(z)$ of \cref{e:f} is analytically continuable to $\tilde\C\setminus \exp(-K)$ and is of polynomial type at most $c$ there.
    \end{alphabetize}
\end{theorem}

We prove the theorem in \Cref{a:proof}, thereby obtaining independent proofs of \Cref{t:carlsonu} with $K$ convex, and its special cases.

When $f$ is already entire, the result remains informative about its behavior at infinity. Such properties do not follow from the power-series representation alone. In this sense, the representation remains local even though it converges throughout $\C$. Ford obtained another result in this direction \cite[§12]{ford1936}, but the present theorem does not imply it. Related results were found by Mavrodi \cite{mavrodi1986,mavrodi1989a}.

Note that $\Phi$ in \Cref{t:carlsonlapl} cannot be defined by \cref{e:Phi} in this setting, as that series may diverge. Instead, a different definition using the integral form of the Laplace transform is required. This approach is outlined in \cite{beauduin2026b}, where we also gave a more specialized continuation theorem that provides more information about the concrete examples discussed there.

\begin{theorem}
Let $K$ be a closed subset of $\bar\C$ such that $\kappa_K(0) < \infty$ and \cref{e:vert} holds. The following are equivalent:
    \begin{alphabetize}
        \item There exists an entire function $\phi(s)$ of order at most $1$ satisfying \cref{e:interp} and such that $\Phi(p)$ is analytically continuable to $\bar\C\setminus K$.
        \item The series $f(z)$ of \cref{e:f} is analytically continuable to $\tilde\C\setminus\exp(-K)$ and there exists a sequence of coefficients $(a_n)_{n\in\N}\subset\C$ such that for all $\varrho > 1$,
        \begin{equation}\label{e:fasymp}
            f(z) - \sum_{n=1}^N \frac{a_n}{z^n} = O\pa{\frac{e^{N^\varrho}}{z^{N+1}}}, \quad z \to \tilde\infty \setminus \exp(-K),
        \end{equation}
        uniformly in $N\in\N$.
    \end{alphabetize}
    Furthermore, $a_n = -\phi(-n)$.
\end{theorem}

Here, the Poincaré/Lindelöf asymptotic expansion \eqref{e:fasymp} is not merely a consequence but an integral part of the theorem.

\section{Analytic continuation to a bounded set}

The preceding continuation regions are unbounded, whether their singularity sets are bounded as in \Cref{s:bounded} or unbounded as in \Cref{s:u}. In 1946, Cowling considered $K = e^{i\alpha}\Pi_{0} \cap e^{i\beta}\Pi_{0}$ where $-\pi < \alpha < 0 < \beta < \pi$ and gave sufficient conditions for continuation of \eqref{e:f} to $\C\setminus\exp(-K)$, a bounded set when $\beta-\alpha<\pi$.\footnote{In Cowling's notation, $\beta-\alpha < \pi$ corresponds to cases A/A' and B \cite[Thms.~1--2]{cowling1946}. Case C, where $-\alpha = \beta = \frac\pi2$, reduces to \Cref{t:lind02}, while cases D/D', where $\beta - \alpha > \pi$, are covered by \Cref{t:carlsonb}.} Agmon rediscovered a special case \cite[Lem.~3']{agmon1952} and supplied a partial converse \cite[Lem.~3]{agmon1952}.\footnote{Agmon attributed these lemmas to Bernstein's theorem in \Cref{s:arc}, but that theorem does not imply them.}

The following theorem, due to Gawronski and Trautner \cite[Thms.~1--2]{gawronski1976}, substantially extends these results by characterizing continuation to a bounded set. See also \cite[Thm.~3.2]{arakelian1992}.

In this section, intervals of arguments are interpreted as oriented intervals on $\Theta = \R/2\pi\Z$: $(\alpha,\beta)$ denotes the counterclockwise arc from $\alpha$ to $\beta$, with the analogous convention for closed and half-open intervals.

To describe bounded continuation sets $\C\setminus\exp(-K)$ using the support function, we exclude the region towards infinity by imposing $\kappa_K(\pi) = \infty$, which implies
\begin{equation}\label{e:u}
    \kappa_K(\theta) = \infty \quad\text{for } \theta\in \pa{\frac\pi2, -\frac\pi2},
\end{equation}
and prevent passage through infinity by imposing
\begin{equation}\label{e:vert3}
    \kappa_K\pa{-\frac\pi2}+\kappa_K\pa{\frac\pi2}=2\pi,
    \quad\text{and}\quad
    -\kappa'_K\pa{-\frac\pi2},\, \kappa'_K\pa{\frac\pi2}<\infty.
\end{equation}

\begin{theorem}[Gawronski--Trautner, 1976]\label{t:gt}
    Let $K$ be a closed convex set satisfying $\kappa_K(0) < \infty$, $\kappa_K(\pi) = \infty$ and \cref{e:vert3}. The following are equivalent:
    \begin{alphabetize}
        \item There exists an entire function $\phi(s)$ of exponential type satisfying \cref{e:interp} and $h_\phi(\theta) \le \kappa_K(-\theta)$ for $\theta\in [-\frac\pi2, \frac\pi2]$.
        \item The series $f(z)$ of \cref{e:f} is analytically continuable to $\C\setminus\exp(-K)$.
    \end{alphabetize}
\end{theorem}

Here the continuation domain is bounded, whereas the possible singularity set $\exp(-K)$ is unbounded. In particular, \Cref{t:gt} supplies a converse to Cowling's cases A, A', and B \cite{cowling1946}.

\subsection{Continuation to an arc of the circle of convergence}\label{s:arc}

The theorems in this subsection concern continuation across an arc of the circle of convergence. The following result is due to Bernstein \cites[§7]{bernstein1930}[Chap.~V, Thm.~III]{bernstein1933}.

\begin{theorem}[Bernstein, 1930]\label{t:bernstein}
    Let $\varsigma\in (0, \pi]$. The following are equivalent:
    \begin{alphabetize}
        \item There exists a function $\phi(s)$ analytic and of exponential type in $S_{[-\eta, \eta]}$, for some $\eta > 0$, satisfying \cref{e:interp} and $h_\phi(\theta) \le (\varsigma - \epsilon) \abs{\sin \theta}$, for $\theta\in (-\eta, \eta)$ and some $\epsilon > 0$.
        \item The series $f(z)$ of \cref{e:f} is analytically continuable to $\D\cup\cbra*{e^{i\theta} : \theta\in [\varsigma, -\varsigma]}$.
    \end{alphabetize}
\end{theorem}

Unlike the preceding results, the continuation set need not be open. The statement above incorporates a correction noted by Eremenko \cite{eremenko2008}: Bernstein's 1933 formulation uses an open arc where a closed one is required, and the version with the open arc is false.

The case $\varsigma=\pi$, namely continuation to $-1$, was known to Carlson \cite[§7]{carlson1914}, who noted that $\phi$ can be assumed entire. When Arakelian and Martirosyan rediscovered the theorem in 1987 \cite[Lem.~1]{arakelian1987}, they strengthened this conclusion by showing that $\phi$ can be chosen to have exponential type. See also \cite{arakelian1988,arakelian1991,eremenko2008}.

The following theorem was announced in \cite[Thm.~I.10]{arakelian1991} and proved by Arakelian in \cite{arakelian1992}; see also \cite{arakelian2007}.

\begin{theorem}[Arakelian, 1991]\label{t:arak91}
    Let $\sigma\in [0, \pi)$. The following are equivalent:
    \begin{alphabetize}
        \item There exists an entire function $\phi(s)$ of exponential type satisfying \cref{e:interp}, $h_\phi(0) \le 0$ and $\pm h'_\phi(0^{\pm}) \le \sigma$.
        \item The series $f(z)$ of \cref{e:f} is analytically continuable to $\D\cup\cbra*{e^{i\theta} : \theta\in (\sigma, 2\pi-\sigma)}$.
    \end{alphabetize}
\end{theorem}

By the correspondence with support functions discussed in \Cref{s:dufr}, indicator functions admit left and right derivatives at every point.

The indicator conditions in \Cref{t:bernstein,t:arak91} are closely related. When $h_\phi(0)=0$, the bound in \Cref{t:bernstein}, on a sufficiently small sector, is equivalent to $\pm h'_\phi(0^{\pm})<\varsigma$. When $h_\phi(0)<0$, that bound holds on a sufficiently small sector without any restriction on these derivatives. The strict derivative inequalities in the normalized case correspond to continuation across a closed arc, whereas the non-strict inequalities in \Cref{t:arak91} characterize continuation across an open arc. Consequently, the two theorems are equivalent, a fact Arakelian did not note when stating them separately.

\subsection{Relation to overconvergence}

The notion of \emph{overconvergence} \cite[§90]{dienes1957} is particularly relevant here. A power series is overconvergent if a subsequence of its partial sums converges uniformly on compact subsets of a domain larger than its disk of convergence, namely in a neighborhood of an arc of the boundary.

Assume that the power series $f$ of \cref{e:f} has a finite, nonzero radius of convergence $R$. By the Cauchy--Hadamard theorem,
\begin{equation}
    \inv R = \limsup_{n\to\infty} |c_n|^{1/n}.
\end{equation}
Consider the following five assertions:
\begin{itemize}
    \item[(OC)] The series $f$ is overconvergent, that is, a subsequence of its partial sums converges uniformly on compact subsets of a domain larger than the disk of convergence.
    \item[(AC)] $f$ is analytically continuable to an arc of its circle of convergence.
    \item[(H--O)] The power series $f$ has \emph{Hadamard--Ostrowski gaps}, that is, there exist two increasing sequences $(p_k)_k$ and $(q_k)_k$ of positive integers such that $p_k<q_k<p_{k+1}\to\infty$, and
    \[
    \limsup_{k\to\infty} \frac{p_k}{q_k} < 1 \qquad\text{and}\qquad \limsup_{\substack{n\to\infty\\n\in\bigcup_k(p_k,q_k)}} |c_n|^{1/n}< \inv{R}.
    \]
    \item[(I)] There exists an entire function $\phi(s)$ of exponential type satisfying \cref{e:interp}, $h_\phi(0)=-\log R$ and
    \begin{equation}\label{e:hvert2}
        h'_\phi(0^+) - h'_\phi(0^-) < 2\pi.
    \end{equation}
    \item[($\neg$CRG)] The function $\phi$ in (I) is \textbf{not} of \emph{completely regular growth} on the positive real ray in the sense of Levin and Pfluger \cite{levin1996}.
\end{itemize}

\Cref{e:hvert2} can be interpreted as a local version of \cref{e:hvert1}. In effect, viewing $h_\phi$ as the support function of a compact convex subset $K$ of $\C$, \cref{e:hvert2} says that the vertical right edge of $K$ has length less than $2\pi$. Then, as a corollary of \Cref{t:arak91},
\[
(\text{AC}) \iff (\text I).
\]
Moreover, the arc of continuation contains the points whose arguments lie in $(h'_\phi(0^+), h'_\phi(0^-))$.

Ostrowski's first \cite{ostrowski1921} and second \cite{ostrowski1923} overconvergence theorems together yield
\[
(\text{AC}) \wedge (\text{H--O}) \iff (\text{OC}),
\]
and the convergent subsequence is indexed by $(p_k)_k$; see also \cites[Chap.~XI]{dienes1957}{mushenheim1963}.

Mavrodi sought a more practical characterization of overconvergence by relating it to completely regular growth \cite{mavrodi1985,mavrodi1989,mavrodi1989a}.
\begin{theorem}[Mavrodi, 1985]\label{t:mavrodi}
    Under {\upshape (I)},
    \begin{equation}\label{e:mavrodi}
        (\neg\mathrm{CRG}) \iff (\mathrm{OC}).
    \end{equation}
\end{theorem}
This result solves a problem posed by Mushenheim and Macintyre \cite[p.~225]{mushenheim1963}, namely to describe overconvergence in terms of the coefficient function. The proof relates Hadamard--Ostrowski gaps to failure of completely regular growth on the positive real ray through a result of Azarin \cite{azarin1966}. \Cref{t:mavrodi} also contains a consequence outlined by Mavrodi: if one function $\phi$ satisfying (I) fails to have completely regular growth on the positive real ray, then overconvergence follows from \Cref{t:mavrodi}; applying the converse implication to any other admissible interpolant shows that it too fails to have completely regular growth on that ray.

\section{Analytic continuation to a Riemann surface}

The preceding results concern subsets of the complex plane, the simplest Riemann surface. The first analogous results on the Riemann surface of the logarithm include the following theorem of Leau \cite[Chap.~II, §12--15]{leau1899}.

\begin{theorem}[Leau, 1899]\label{t:leausurf}
    If $\phi(s)$ is analytic at $\hat\infty$ and satisfies \cref{e:interp}, then the series $f(z)$ of \cref{e:f} is analytically continuable along every path avoiding $1$ and non-principal lifts of $\bar\D$.
\end{theorem}

Le Roy then strengthened Leau's result \cite[§22]{leroy1900}.\footnote{Leau's phrasing is ambiguous, so it is unclear whether he claimed \Cref{t:leroysurf}. His proof, however, establishes \Cref{t:leausurf} \cite{ostrowski1933,bieberbach1955}.}

\begin{theorem}[Le Roy, 1900]\label{t:leroysurf}
    If $\phi(s)$ is analytic at $\hat\infty$ and satisfies \cref{e:interp}, then the series $f(z)$ of \cref{e:f} is analytically continuable along every path avoiding $1$ and non-principal lifts of $0$.
\end{theorem}

\begin{example}
    The function $\phi(s) = 1/(s+1)$ provides a simple example of \Cref{t:leroysurf}. It is associated with the multivalued function $f(z) = -\log(1-z)/z$. On non-principal sheets, the monodromy of the logarithm creates a simple pole at $0$, illustrating the singularities that can arise there.
\end{example}

Unaware of Le Roy, Faber conjectured \Cref{t:leroysurf} on the basis of Leau's results \cite{faber1903}. Ostrowski subsequently gave an alternative proof \cite{ostrowski1933}. More general theorems, however, had already been obtained by Lindelöf \cites[§15]{lindelof1902}{lindelof1903}[§63]{lindelof1905}.

\begin{theorem}[Lindelöf, 1902]\label{t:lindRS02}
    If $\phi(s)$ is analytic and of exponential type $0$ at $\tilde\infty e^{i \R}$ and satisfies \cref{e:interp}, then $f(z)$ defined by \cref{e:f} is analytically continuable along every path avoiding $1$ and non-principal lifts of $0$.
\end{theorem}

Here, $\tilde\infty e^{i \R}$ denotes the radial infinities on the Riemann surface of the logarithm, where $\R$ replaces the set of angles $\Theta$. Accordingly, for a function $h$ on $\R$, set $\W(h):=\bigcap_{\theta\in\R}e^{i\theta}\Pi_{h(\theta)}$ as a subset of the logarithmic coordinate plane. In these terms, Lindelöf obtained the following extension in terms of the indicator function \eqref{e:indic} \cite[§64]{lindelof1905}.

\begin{theorem}[Lindelöf, 1905]
    If $\phi(s)$ is analytic and of exponential type less than $\pi$ at $\tilde\infty e^{i\R}$ and satisfies \cref{e:interp}, then $f(z)$ defined by \cref{e:f} is analytically continuable along every path avoiding $\exp(-\W(h_\phi)^\dagger)$ and non-principal lifts of $0$.
\end{theorem}

The final result of this section is due to Arakelian and Schmieder \cite{arakelian2001}. Its proof again uses approximation, but Carlson's theorem now gives genuine uniqueness of $\phi$.

\begin{theorem}[Arakelian--Schmieder, 2001]\label{t:a-s}
    The following are equivalent:
    \begin{alphabetize}
        \item There exists a function $\phi(s)$ analytic and of exponential type $0$ at $\tilde\infty e^{i(-\pi, \pi)}$ that satisfies \cref{e:interp}.
        \item The series $f(z)$ of \cref{e:f} is analytically continuable along every path avoiding $1$ and non-principal lifts of $\bar\D$, and is of polynomial type at $\tilde\infty e^{i(-\frac\pi2,\frac\pi2)}$.
    \end{alphabetize}
\end{theorem}
The theorem can be viewed as a converse to \Cref{t:leausurf}.

\section{Further directions and references}

Classical continuation results have been extended to functions of two \cite{ford1906,janusauskas1977,janusauskas1977a,janusauskas1980} and several complex variables \cite{mkrtchyan2015,mkrtchyan2019,mkrtchyan2022,mkrtchyan2024,ivanova2026}. Other work treats meromorphic interpolants \cite{ford1916,kampedeferiet1926,stein1965,selberg1978,gawronski1981,mkrtchyan2015}, Banach-space-valued interpolants \cite{arakelian1995,arakelian2003}, and functional-analytic counterparts to continuation results \cite{stein1965,demicheli1999,demicheli2012}.

There is a certain elegance in the paired expansions \cref{e:fphi(n),e:fphi(-n)} associated with \Cref{t:dufr}. This symmetry is lost in \Cref{s:u}, where the second expansion becomes asymptotic rather than convergent. A result of Bernstein \cite{bernstein1928,bernstein1928a,bernstein1928b} extending \Cref{t:leroylindconv} partially restores it by allowing $\phi$ to have at most factorial-like growth and requiring the power series $f$ to be summable, rather than convergent, in the Mittag-Leffler sense \cite[§4.11]{hardy1949}. This raises the question whether the results in \Cref{s:ulapl} admit corresponding extensions without the assumption $\kappa_K(0)<\infty$.

In another direction, Bernstein originally presented \Cref{t:bernstein} for generalized Dirichlet series \cite{bernstein1930,bernstein1933}. Analytic continuability of power series has also found a recent application to the continuation of corresponding Dirichlet series \cite{navas2021}.

More broadly, the unified formulations developed here raise the question whether a single theorem can encompass all the continuation results surveyed in this paper, including those for bounded and unbounded continuation domains and for the Riemann surface of the logarithm. Such a theorem would require a common treatment of the different compactifications and growth conditions that arise.

\subsection*{}

\appendix
\appendixpage

\section{Proof of the refinement of Carlson's second continuation theorem}\label{a:proof}

\begin{proof}[of \Cref{t:carlsonlapl}]
    Assume (a). The Laplace transform is defined outside $\W(h_\phi)^\dagger$ by \cite{beauduin2026b}
    \begin{equation}
        \Phi(p) = \int_0^{\infty e^{i\theta}} \phi(s) e^{-ps} \d s, \quad \Re(p e^{i\theta}) > h_\phi(\theta), \quad \theta\in \bra{-\frac\pi2,\frac\pi2}.
    \end{equation}
    To obtain decay towards $\Re p=-\infty$, fix $\gamma \in (c,0)$ and set
    \[
    \Phi_\gamma(p) := \Phi(p)-\int_0^\gamma\phi(s)e^{-ps}\d s.
    \]
    Shifting the integration ray gives, for $\Re(\pm ip)>h_\phi(\pm\frac\pi2)$,
    \[
    \Phi_\gamma(p) = \int_\gamma^{\gamma\pm i\infty}\phi(s)e^{-ps}\d s.
    \]
    On each smaller closed half-plane, this yields
    \[
    |\Phi_\gamma(p)| \le \int_0^\infty |\phi(\gamma\pm it)|e^{-\Re(p(\gamma\pm it))}\d t \ll e^{-\gamma\Re p}.
    \]
    The correction is entire and tends to zero as $\Re p\to-\infty$, so it changes neither the finite singularities nor those towards $-\infty$.

    Work locally near any $z\notin\exp(-K)\cup\{0\}$ and choose a branch of $\log$. The contour theorem of \cite{beauduin2026b} gives a contour $\Gamma$ around $K$ that leaves the poles of $(1-ze^p)^{-1}$ outside for all $z$ in this neighborhood. This construction also covers disconnected and purely infinite sets $K$.

    The unbounded parts of $\Gamma$ lie in half-planes where the preceding estimate holds. On the compact remainder, the estimate follows by enlarging the constant. Laplace inversion gives the absolutely convergent integral
    \begin{equation}
        \phi(s) = \inv{2\pi i}\oInt_{\Gamma}\Phi_\gamma(p)e^{ps}\d p,
        \quad \Re s>\gamma,
    \end{equation}
    which we sum at $s=n$ with weights $z^n$ for small $|z|$, using Fubini's theorem. Analytic continuation and the substitution $p\mapsto p+\log z$ then give
    \[
    f(z) = \inv{2\pi i}\oInt_{\Gamma+\log z}\frac{\Phi_\gamma(p-\log z)}{1-e^p}\d p.
    \]
    Where two such formulas apply, deforming one contour into the other without crossing singularities shows that they give the same value of $f(z)$. Together with the original power series near $0$, they give the continuation to $\C\setminus\exp(-K)$.

    It remains to bound $f$ at infinity. The estimate for $\Phi_\gamma$ gives
    \[
    |f(z)|\ll |z|^\gamma\oInt_{\Gamma+\log z}\frac{e^{-\gamma\Re p}}{|1-e^p|}\abs{\d p}.
    \]
    Fix a closed set of admissible directions. As $z$ tends to infinity in these directions, the contours can be chosen uniformly so that their translates by $\log z$ stay separated from $2\pi i\Z$. The denominator is therefore bounded away from zero when $\abs{\Re p}\le1$.

    On the remaining parts, use $x=\Re p$ with $\abs{\d p}\ll\abs{\d x}$. The integrand is bounded by a constant times $e^{-\gamma x}$ for $x\le-1$ and $e^{-(\gamma+1)x}$ for $x\ge1$. Both bounds are integrable since $-1<\gamma<0$. Thus $f(z)=O(|z|^\gamma)$ for every $\gamma\in(c,0)$, proving (b).

    \medskip

    Conversely, assume (b) and define
    \[
    \phi(s) := \inv{2\pi i}\oInt_\Gamma f(e^{-u}) e^{us} \d u,
    \]
    where $\Gamma$ encircles $K$ in a horizontal strip of width less than $2\pi$, with horizontal tails towards $\Re u=-\infty$. The growth assumption on $f$ gives, for every $\gamma \in (c, 0)$,
    \begin{align*}
        |\phi(s)|
        &\le \inv{2\pi}\oInt_\Gamma |f(e^{-u})|e^{\Re(us)}\abs{\d u} \\
        &\le \frac{C_\gamma}{2\pi}\oInt_\Gamma
        e^{-\gamma\Re u}e^{\Re u\Re s-\Im u\Im s}\abs{\d u}.
    \end{align*}
    Since $\Im u$ is bounded on $\Gamma$, the integral converges absolutely and locally uniformly for $\Re s>\gamma$ for every $\gamma > c$. The same estimate gives an exponential bound on every smaller half-plane, which shows that $\phi$ is analytic and of exponential type for $\Re s>c$.

    For $\Re p$ sufficiently large, Fubini's theorem gives
    \[
    \Phi(p) = \inv{2\pi i}\oInt_\Gamma\frac{f(e^{-u})}{p-u}\d u,
    \]
    which gives the continuation of $\Phi$ to $\bar\C\setminus K$ after deforming $\Gamma$ towards $K$.
    
    To verify interpolation, use $z=e^{-u}$, which is injective on the strip containing $\Gamma$. This gives
    \[
    \phi(s) = -\inv{2\pi i}\oInt_{\partial U}\frac{f(z)}{z^{s+1}}\d z,
    \]
    where $U$ is a neighborhood of $\exp(-K)$ whose boundary tails are radial and whose closure avoids $0$. Choose $\epsilon>0$ so that $\D_\epsilon$ does not intersect $U$. For all sufficiently large $R$ and every $n\in\N$, Cauchy's integral formula gives
    \[
    \oint\limits_{\partial(U\cup\D_R)}\frac{f(z)}{z^{n+1}}\d z = \oInt_{\partial\D_\epsilon}\frac{f(z)}{z^{n+1}}\d z+\oInt_{\partial U}\frac{f(z)}{z^{n+1}}\d z = 2\pi i(c_n-\phi(n)),
    \]
    whose left-hand side tends to zero. Indeed, for any $\gamma\in(c,0)$,
    \begin{align*}
        \abs{\,\oint\limits_{\partial(U\cup\D_R)}\frac{f(z)}{z^{n+1}}\d z}
        &\le \oint\limits_{\partial(U\cup\D_R)}\frac{|f(z)|}{|z|^{n+1}}\abs{\d z} \\
        &\ll R^{\gamma-n}+\int_R^\infty r^{\gamma-n-1}\d r \\
        &\ll R^{\gamma-n}\to_{R\to\infty} 0.
    \end{align*}
    Hence $\phi(n)=c_n$.
\end{proof}

The converse construction also gives the indicator bound in \Cref{t:carlsonu}. Indeed,
\begin{align*}
    |\phi(re^{i\theta})|
    &\le \inv{2\pi}\oInt_\Gamma |f(e^{-u})|e^{\Re(ure^{i\theta})}\abs{\d u} \\
    &\le \frac{C_\gamma}{2\pi}\oInt_\Gamma e^{-\gamma\Re u}\abs{\d u}
    \exp\pa{r\sup_{u\in\Gamma}\Re(ue^{i\theta})},
\end{align*}
and choosing $\Gamma$ arbitrarily close to $K$ yields $h_\phi(\theta)\le\kappa_K(-\theta)$ for $\theta\in[-\frac\pi2,\frac\pi2]$. Thus \Cref{t:carlsonlapl} recovers \Cref{t:carlsonu}.

\section*{Acknowledgements}

The author used OpenAI Codex to review drafts of this paper and to suggest improvements to its language, exposition and presentation. All suggestions were assessed by the author, who independently verified the mathematical content and takes full responsibility for the final text.

\addtocategory{add}{phragmén1908,riesz1920,ostrowski1921,ostrowski1923,boas1954,hardy1949,mushenheim1963,bessis1965,azarin1966,newman1976,conway1978,bros1996,flajolet2010,navas2021}

\printbibliography[
  title={Survey references},
  notcategory=add,
]

\printbibliography[
  title={Additional references},
  category=add,
  heading=subbibliography,
]

\end{document}